\documentclass[11pt]{article}
\usepackage{amsmath,amsfonts,amsthm,amssymb,color}
\usepackage{fancybox}

\newtheorem{theorem}{Theorem}[section]
\newtheorem{corollary}[theorem]{Corollary}
\newtheorem{lemma}[theorem]{Lemma}

\newtheorem{proposition}[theorem]{Proposition}
\newtheorem{claim}[theorem]{Claim}
\newtheorem{fact}[theorem]{Fact}

\theoremstyle{definition}
\newtheorem{definition}[theorem]{Definition}

\newenvironment{fminipage}%
  {\begin{Sbox}\begin{minipage}}%
  {\end{minipage}\end{Sbox}\fbox{\TheSbox}}

\newenvironment{algbox}[0]{\vskip 0.2in
\noindent 
\begin{fminipage}{6.3in}
}{
\end{fminipage}
\vskip 0.2in
}

\def\pgeq{\succcurlyeq}
\def\pge{\succ}

\def\Approx#1{\approx_{#1}}

\def\bvec#1{{\mbox{\boldmath $#1$}}}

\def\defeq{\stackrel{\mathrm{def}}{=}}

\def\ceil#1{\left\lceil #1 \right\rceil}

\def\norm#1{\left\| #1 \right\|}

\newcommand\bb{\boldsymbol{\mathit{b}}}

\newcommand\ee{\boldsymbol{\mathit{e}}}

\newcommand\uu{\boldsymbol{\mathit{u}}}
\newcommand\vv{\boldsymbol{\mathit{v}}}
\newcommand\ww{\boldsymbol{\mathit{w}}}

\newcommand\xx{\boldsymbol{\mathit{x}}}

\renewcommand\AA{\boldsymbol{\mathit{A}}}
\newcommand\BB{\boldsymbol{\mathit{B}}}

\newcommand\DD{\boldsymbol{\mathit{D}}}

\newcommand\II{\boldsymbol{\mathit{I}}}

\newcommand\MM{\boldsymbol{\mathit{M}}}
\newcommand\LL{\boldsymbol{\mathit{L}}}

\newcommand\WW{\boldsymbol{\mathit{W}}}
\newcommand\VV{\boldsymbol{\mathit{V}}}
\newcommand\XX{\boldsymbol{\mathit{X}}}
\newcommand\YY{\boldsymbol{\mathit{Y}}}

\newcommand\ZZ{\boldsymbol{\mathit{Z}}}

\newcommand\AAhat{\boldsymbol{\widehat{\mathit{A}}}}
\newcommand\AAapprox{\boldsymbol{\widetilde{\mathit{A}}}}
\newcommand\DDhat{\boldsymbol{\widehat{\mathit{D}}}}
\newcommand\DDapprox{\boldsymbol{\widetilde{\mathit{D}}}}
\newcommand\LLhat{\boldsymbol{\widehat{\mathit{L}}}}
\newcommand\LLapprox{\boldsymbol{\widetilde{\mathit{L}}}}
\newcommand\MMhat{\boldsymbol{\widehat{\mathit{M}}}}

\newcommand\AAtil{\boldsymbol{\widetilde{\mathit{A}}}}
\newcommand\DDtil{\boldsymbol{\widetilde{\mathit{D}}}}

\newcommand\xxtil{\boldsymbol{\tilde{\mathit{x}}}}

\newcommand\Otil{\widetilde{O}}

\begin{document}

\title{
An Efficient Parallel Solver for SDD Linear Systems
}

\author{
Richard Peng
\thanks{Part of this work was done while at CMU and was supported by a Microsoft Research PhD Fellowship}\\
M.I.T.\\
rpeng@mit.edu
\and
Daniel A. Spielman 
\thanks{Supported in part by AFOSR Award FA9550-12-1-0175, a Simon's Investigator Grant, and a MacArthur Fellowship.}
\\ 
Yale University\\
spielman@cs.yale.edu
}

\maketitle

\begin{abstract}
We present the first parallel algorithm for solving systems of linear equations in 
  symmetric, diagonally dominant (SDD) matrices that runs in
  polylogarithmic time and nearly-linear work.
The heart of our algorithm is a construction of a sparse approximate
  inverse chain for the input matrix: a sequence of sparse matrices
  whose product approximates its inverse.
Whereas other fast algorithms for solving systems of equations in SDD matrices
  exploit low-stretch spanning trees, our algorithm only requires
  spectral graph sparsifiers.
\end{abstract}

\thispagestyle{empty}
\newpage
\setcounter{page}{1}

\section{Introduction}\label{sec:intro}

The problem of solving systems of linear equations in 
  symmetric, diagonally dominant (SDD) matrices, and in particular the Laplacian matrices of graphs,
  arises in applications ranging from the solution of elliptic
  PDEs~\cite{BomanHendricksonVavasis}, to the computation of maximum
  flows in graphs~\cite{DaitchSpielman,ChristianoEtAl,madryFlow},
  to semi-supervised learning~\cite{ZhuGL,ZhouFramework,ZhouBLWS03}.
It has recently been exploited as a primitive in many other algorithms \cite{kelnerMadryPropp,KelnerMillerPeng,OSV,LevinKoutisPeng,koutis2011combinatorial}.

In the last decade, there have been remarkable advances in the development of fast algorithms
  for solving systems of linear equations in SDD matrices.
Following an approach suggested by Vaidya~\cite{Vaidya}, Spielman and Teng~\cite{SpielmanTengLinsolve}
  developed a nearly-linear time algorithm for this problem.
Their algorithm has three main ingredients: a multi-level framework
  suggested by Vaidya~\cite{Vaidya}, Joshi~\cite{Joshi} and
  Reif~\cite{Reif};
  low-stretch spanning tree preconditioners, introduced by Boman and
  Hendrickson~\cite{BomanHendricksonAKPW}
  as an improvement of Vaidya's maximum spanning tree preconditioners;
  and spectral graph sparsifiers~\cite{SpielmanTengSparsifier}.
Koutis, Miller and Peng~\cite{KMP1,KMP2} developed a simpler and
  faster way to exploit these ingredients, resulting in an algorithm for
  finding $\epsilon$-approximate solutions to these systems in
  time $\Otil  (n \log n \log \epsilon^{-1})$.

Kelner, Orecchia, Sidford and Zhu~\cite{KOSZ} discovered a simple and efficient algorithm
  for solving these linear equations that relies only on low-stretch
  spanning trees, avoiding the use of graph sparsifiers and the multi-level framework.
This algorithm has been improved by Lee and Sidford~\cite{LeeSidford}
  to run in time $\Otil (n \log^{3/2} n \log \epsilon^{-1})$.

There has also been some progress in parallelizing these solvers.
When describing  a parallel algorithm, we call the number of operations performed
  its \textit{work} and the time it takes in parallel its \textit{depth}.
For planar graphs, Koutis and Miller gave an algorithm that requires
nearly-linear work and depth close to $m^{1/6}$~\cite{KoutisMiller}.
Their approach was extended to general graphs by Blelloch, \textit{et. al.},
who gave parallel algorithms for constructing low-stretch embeddings,
leading to an algorithm with depth close to $m^{1/3}$  \cite{blelloch2011near}.

We present the first algorithm that requires nearly-linear work and
poly-logarithmic depth.  
Unlike the previous nearly-linear time algorithms for solving systems
  in SDD matrices, our algorithm does not require low-stretch spanning trees.
It merely requires spectral sparsifiers of graphs, which we use 
to construct a \textit{sparse approximate inverse chain}.
It would be easy to solve a system of equations in a matrix if one
  had a sparse matrix that approximates its inverse.
While not all matrices have sparse approximate inverses, we construct something
  almost as good for SDD matrices: a sequence of sparse matrices whose product
  approximates the inverse.

Given a sparse approximate inverse chain, the resulting algorithm for
  solving the linear equations is very simple:
  we multiply a vector by each matrix in the chain twice.
The reason that we multiply by each twice, rather than once, is to
  make the solver a symmetric operator.
This is explained in more detail in Section~\ref{sec:overview}.
The algorithm is analogous to the V-cycle used in the Multigrid method~\cite{MultigridBook},
  whereas the multilevel approach of
  \cite{SpielmanTengLinsolve,KMP1,KMP2} required the solution of the
  smallest systems many times, and better resembles the Multigrid W-cycles\footnote{%
Variants of the Multigrid algorithm are among the most popular approaches to solving large
  SDD linear systems in practice.
The advantage of these algorithms is that they can be efficiently parallelized.
However, there is no analysis proving that they should work in general.
In fact, the theory suggests that V-cycles should be sufficient, while the more
  expensive W-cycles are often necessary.
}
.

The work and depth of our solvers depend on the
  accuracy desired and the
  condition number of the input matrix,
  which is defined in the next section.
As is standard, we say that $\xxtil$ is an $\epsilon$-approximate solution to
  $\MM \xx =\bb$ if $\norm{\xx - \xxtil}_{\MM} \leq \epsilon \norm{\xx}_{\MM},$
  where $\norm{\xx}_{\MM} = (\xx^{T} \MM \xx)^{1/2}$.
We present two types of results about solving linear systems with sparse
  approximate inverse chains.
The first is a proof that they exist and can be constructed in
  polynomial time.

\begin{theorem}
\label{thm:mainExistence}
There is a polynomial time algorithm that on input
 an $n$-dimensional SDDM matrix $\MM$ with $m$ nonzeros and condition number at most $\kappa$, 
 produces a sparse approximate inverse chain that can be used to solve
  any linear equation in $\MM$ to any precision $\epsilon$ in
  work $O ((m + n \log^{3} \kappa ) \log 1/\epsilon)$
  and depth $O (\log n \log \kappa \log 1/\epsilon)$.
\end{theorem}

The second is a proof that slightly weaker approximate inverse chains exist
  and can be constructed in nearly-linear work and polylogarithmic time.

\begin{theorem}
\label{thm:mainEfficient}
There is an algorithm that on input
 an $n$-dimensional SDDM matrix $\MM$ with $m$ nonzeros and condition number at most $\kappa$, 
 produces with probability at least $1/2$ a sparse approximate inverse chain that can be used to solve
  any linear equation in $\MM$ to any precision $\epsilon$ in
  work $O ((m + n \log^{c} n \log^{3} \kappa ) \log 1/\epsilon)$
  and depth $O (\log n \log \kappa \log 1/\epsilon)$, for some constant $c$.
Moreover, the algorithm runs in work $O (m \log^{c_{1}} n \log^3{\kappa})$ and depth $O (\log^{c_{2}} n \log{\kappa})$
  for some other constants $c_{1}$ and $c_{2}$.
\end{theorem}

We present an overview of our algorithm in Section~\ref{sec:overview},
  and supply the details later.

\section{Background and Notation}\label{sec:background}

Throughout this paper we will consider symmetric matrices.
A symmetric matrix is \textit{diagonally dominant}
  if each diagonal entry is at least%
\footnote{This is sometimes called \textit{weak} diagonal dominance, with
\textit{strict} diagonal dominance referring to the case in which each diagonal
is strictly larger.}
 as large as the
  sum of the absolute values of the other entries in its row.
The problem of solving systems of linear equations in symmetric,
  diagonally dominant (SDD) matrices can be reduced to the problems
  of solving systems in either Laplacian matrices---SDD matrices with
  zero rows sums and non-positive off-diagonal elements---or SDDM
  matrices---positive definite SDD matrices with non-positive off-diagonal elements.
The reductions from one form to another approximately preserve the condition
  numbers of the system (see Appendix \ref{sec:condition}).

A fundamental property of a matrix is its \textit{condition number}:
  the ratio of its largest eigenvalue to its smallest.
In the case of a Laplacian matrix, which has $0$ as an eigenvalue,
we consider the ratio to the second-smallest.
The condition number measures how much the solution to a linear system
  in a matrix can change if small changes are made to the target
  vector or the matrix itself.
Thus, the condition number is a lower bound on the precision required when
  carrying out computations with a matrix.
It is also appears in the running time of many numerical algorithms,
  including the one we present here.
The condition number of the Laplacian matrix of a connected weighted graph with $n$ vertices
  is at most $O (n^{3} w_{max} / w_{min})$, where $w_{max}$ and
  $w_{min}$ are the largest and smallest weights of edges in the graph
  (see \cite[Lemma 6.1]{SpielmanTengLinsolve})
The condition number of a submatrix of a Laplacian is at most
  $O (n^{4} w_{max} / w_{min})$
  (see Appendix \ref{sec:condition}).
Our algorithms will run in time depending upon the logarithm of the
  condition number.

We recall that a symmetric matrix $\XX$ is positive definite if all of its eigenvalues are positive.
We indicate this through the notation $\XX \pge \bvec{0}$, where $\bvec{0}$ indicates
  the all-0 matrix.
We similarly write $\XX \pgeq \bvec{0}$ if $\XX$ is positive semidefinite.
For matrices $\XX$ and $\YY$, we write $\XX \pge \YY$ or $\XX \pgeq \YY$
 to indicate that $\XX - \YY$ is positive definite or positive semidefinite.

This partial ordering allows us to define a notion of approximation for matrices.
We write $\XX \Approx{\epsilon} \YY$ to indicate that
\[
 \exp\left(\epsilon\right) \XX \pgeq \YY \pgeq \exp\left(-\epsilon\right) \XX ,
\]
and we remark that this relation is symmetric.
While it is more common to use $1+\epsilon$ in place of $\exp (\epsilon)$ above,
  the exponential facilitates combining approximations.
The $\epsilon$s will be small everywhere in this paper, 
  and we note that when $\epsilon$ is close to $0$, $\exp(\epsilon)$
is between $1+\epsilon$ and $1+2 \epsilon$.
The following facts about the positive definite order and approximation
  are standard (see \cite{SupportGraph,SupportTheory,SpielmanTengLinsolve}).

\begin{fact}\label{fact}
For positive semidefinite matrices
$\XX$, $\YY$, $\WW$ and $\ZZ$,
\begin{enumerate}
\item [a.] if $\YY \Approx{\epsilon} \ZZ$,
then $\XX + \YY \Approx{\epsilon} \XX + \ZZ$;

\item [b.] if  $\XX \Approx{\epsilon} \YY$ and $\WW \Approx{\epsilon} \ZZ$,
  then $\XX + \WW \Approx{\epsilon} \YY + \ZZ $;

\item [c.]
if $\XX \Approx{\epsilon_1} \YY$ and
$\YY \Approx{\epsilon_2} \ZZ$,
then $\XX \Approx{\epsilon_1 + \epsilon_2} \ZZ$;

\item [d.]
if $\XX$ and $\YY$ are positive definite matrices
  such that $\XX \Approx{\epsilon} \YY$,
  then $\XX^{-1} \Approx{\epsilon} \YY^{-1}$;

\item [e.]
if $\XX \Approx{\epsilon} \YY$
and $\VV$ is a matrix, then
$  \VV^T \XX \VV \Approx{\epsilon} \VV^T \YY \VV.$
\end{enumerate}
\end{fact}

\section{Overview}
\label{sec:overview}

Our algorithm is most naturally described for SDDM matrices.
They are the matrices $\MM$ that can be written as
\[
  \MM = \DD - \AA ,
\]
where $\DD$ is a nonnegative diagonal matrix, $\AA$ is a symmetric nonnegative matrix, and $\DD \pge \AA$.

Our algorithm is inspired by the following identity, which holds for
  matrices $\AA$ of norm less than $1$.
\[
  (\II - \AA)^{-1} = \sum_{i \geq 0} \AA^{i} = 
  \prod_{k \geq 0} (\II + \AA^{2^{k}}).
\]
As the terms corresponding to larger values of $k$ make shrinking contributions,
  it is possible to obtain a good approximation to $(\II - \AA)^{-1}$
  by truncating this product.

We exploit a variation of this identity that
  is symmetric, and thus allows us to employ sparsifiers%
\footnote{The product of sparse approximations of matrices need not be an
  approximation of the product, unless the products are taken symmetrically.
}.
Our identity also applies more naturally to SDDM matrices.
It is:
\begin{equation}\label{eqn:identity}
  (\DD - \AA)^{-1}
 =
 \frac{1}{2}
\left[
  \DD^{-1}
 + 
  (\II + \DD^{-1} \AA)
  (\DD - \AA \DD^{-1} \AA)^{-1}
  (\II + \AA \DD^{-1})
 \right].
\end{equation}
This tells us that, aside for a few multiplications of a vector by a matrix,
   we can reduce the problem of solving a linear equation in $\DD - \AA$
  to the problem of solving a linear equation in $\DD - \AA \DD^{-1} \AA$.
Moreover, this latter matrix is also a SDDM matrix (Proposition~\ref{prop:exactGood}).
We can in turn reduce this to the problem of solving a system of equations in
\[
  \DD - (\AA \DD^{-1} \AA) \DD^{-1} (\AA \DD^{-1} \AA ) 
= 
  \DD - \DD (\DD^{-1} \AA)^{4}.
\]
After applying this identity $k$ times, we obtain the matrix
\[
  \DD - \DD (\DD^{-1} \AA)^{2^{k}}.
\]
As $k$ grows large, it becomes
 easy to approximately solve systems of linear equations in this matrix because
  all of the eigenvalues of $\DD^{-1} \AA$ are less than $1$ in absolute value.
So, the term $(\DD^{-1} \AA)^{2^{k}}$ becomes negligible and the matrix approaches $\DD$
  (Corollary \ref{cor:chainShort}).

However, this does not yet yield an efficient algorithm as it requires us
  to multiply vectors by matrices of the form
\[
  \II + (\DD^{-1} \AA)^{2^{k}}
\]
for several integers $k$.
This can require a lot of work if those matrices are dense.
We overcome this problem by sparsifying those matrices.

We now describe this process in more detail.
Let $\MM_{0} = \DD_{0} - \AA_{0}$ be the matrix that defines the equation we want to solve.
We observe (in Proposition~\ref{prop:exactGood}) that the matrix $\DD_{0} - \AA_{0} \DD_{0}^{-1} \AA_{0}$
  is also positive definite and diagonally dominant.
So, a spectral sparsifier~\cite{SpielmanTengSparsifier,BSS} will allow us to 
  approximate this matrix by a sparse SDDM matrix $\DD_{1} - \AA_{1}$.
Proceeding in this fashion, we obtain a sequence of sparse matrices
\[
  \MM_{i} = \DD_{i} - \AA_{i}
\]
such that
\begin{equation}\label{eqn:overviewApprox}
  (\DD_{i} - \AA_{i})^{-1}
 \approx
 \frac{1}{2}
\left[
  \DD_{i}^{-1}
 + 
  (\II + \DD_{i}^{-1} \AA_{i})
  (\DD_{i+1} - \AA_{i+1})^{-1}
  (\II + \AA_{i} \DD_{i}^{-1} )
 \right]
\end{equation}
and so that $\AA_{i}$ becomes negligible as $i$ grows.
These provide a natural algorithm for solving a system of equations 
  of the form $\MM_{0} \xx = \bb_{0}$, which we now present.

\begin{algbox}
$\xx = \textsc{Solve}((\MM_0, \DD_0) \ldots (\MM_d, \DD_d),
\bb)$
\begin{enumerate}
\item [1.] For $i$ from $1$ to $d$, set $  \bb_{i} =  (\II + \AA_{i} \DD_{i}^{-1} ) \bb_{i-1}.$

\item [2.] Set $\xx_{d} = \DD_{d}^{-1} \bb_{d}$.
\item [3.] For $i$ from $d-1$ downto $0$, set
$  \xx_{i} = \frac{1}{2} (\DD_{i}^{-1} \bb_{i} +  (\II +  \DD_{i}^{-1} \AA_{i} ) \xx_{i+1}).$
\end{enumerate}
\end{algbox}

The work required by \textsc{Solve} is proportional to the total number of nonzero entries
  in the matrices in the chain, and the depth is proportional to $\log n$ times $d$.

If the approximations in \eqref{eqn:overviewApprox} are good enough, then $\MM \xx_{0}$ will
  be a good approximation of $\bb_{0}$.
We now roughly estimate how good these approximations need to be.
We carry out the details in the following sections.
If the condition number of $\MM$ is $\kappa$, then the largest eigenvalue of
  $\DD^{-1} \AA$ at most $1-1/\kappa$ (Proposition \ref{pro:ADinv}).

So, if we take $d = O (\log (\kappa ))$, all of the eigenvalues of
  $(\DD^{-1} \AA)^{2^{d}}$ will be close to 0.
At every level of the recursion, we will incur a loss in approximation quality.
To limit the overall loss to a constant, we will require the approximation
  in \eqref{eqn:overviewApprox} to be at least as good as $\Approx{1/2d}$.
Using the best-known sparsifiers, we could achieve this with sparsifiers that have $O(d^{2} n)$ edges.
So, the total number of nonzero entries in all of the matrices $A_{1}, \dotsc , A_{d}$ will be
  $O (d^{3} n) = O (n \log^{3} \kappa)$.
We establish the existence of such chains in Section~\ref{sec:exist}.
If we instead use sparsifiers that we presently know how to construct efficiently,
we multiply the total number of nonzero entries by a factor of  $O(\log^{c} n)$, for some constant $c$.
In Section~\ref{sec:construct}, we give a nearly linear work, polylog depth
algorithm for constructing these chains.

\section{Approximate Inverse Chains}\label{sec:saic}

\begin{definition}\label{def:saic}
We say that two sequences of matrices $\DD_{1}, \dotsc , \DD_{d}$
  and $\AA_{1}, \dotsc , \AA_{d}$ are an approximate inverse chain
  for a SDDM matrix $\MM_{0} = \DD_{0} - \AA_{0}$ if, 
  for all $i$, $\DD_{i}$ is a nonnegative diagonal matrix, $\AA_{i}$
  is a nonnegative symmetric matrix, $\MM_{i} \defeq \DD_{i} - \AA_{i} \pge \bvec{0}$, and
  there exist numbers $\epsilon_{0}, \dotsc , \epsilon_{d}$ such that
  $\sum_{i=0}^{d} \epsilon_{i} \leq 2$ and
\begin{itemize}
\item [a.] for all $1 \leq i \leq d$, 
  $\MM_{i} \Approx{\epsilon_{i-1}} \DD_{i-1} - \AA_{i-1} \DD_{i-1}^{-1} \AA_{i-1}$;

\item [b.] for all $1 \leq i \leq d$, $\DD_{i} \Approx{\epsilon_{i-1}} \DD_{i-1}$; and,

\item [c.] $\DD_{d} \Approx{\epsilon_{d}} \MM_{d}$.
\end{itemize}
\end{definition}

For the rest of this section, we assume that
  $\DD_{1}, \dotsc , \DD_{d}$
  and $\AA_{1}, \dotsc , \AA_{d}$ are an approximate inverse chain
  for $\DD_{0} - \AA_{0}$.

The choice of the bound of $2$ on $\sum \epsilon_{i}$ is not particularly important:
  any constant will allow us to use an approximate inverse chain to solve systems
  of linear equations in $\DD_{0} - \AA_{0}$.
We remark that many of the matrices $\AA_{i}$ will have nonzero diagonal entries.
For this reason, we keep the representation of $\DD_{i}$ and $\AA_{i}$ separate.

We now show that an approximate inverse chain allows one to crudely solve systems of linear
  equations in $\DD_{0} - \AA_{0}$ in time proportional to the number
  of nonzero entries in the matrices in the chain.
We first verify that if we replace $\DD_{i} - \AA_{i} \DD_{i}^{-1} \AA_{i}$ by  $\MM_{i+1}$ in identity \eqref{eqn:identity},
  then we still obtain a good approximation to $\MM_{i}^{-1}$.

\begin{lemma}\label{lem:levelPrecon}
For $0 \leq i < d$, 
\[
\MM_i^{-1} 
\Approx{\epsilon_{i}} 
\frac{1}{2} \left( \DD_i^{-1} + \left( \II + \DD_i^{-1} \AA_i \right) \MM^{-1}_{i + 1} \left(\II + \AA_i \DD_i^{-1}  \right) \right).
\]
\end{lemma}

\begin{proof}
We have
\begin{align*}
 \MM_{i+1}
& \Approx{\epsilon_{i}}
  \DD_{i} - \AA_{i}  \DD_{i}^{-1} \AA_{i},
\quad  \text{which by Fact \ref{fact}d implies}
\\
 \MM_{i+1}^{-1}
& \Approx{\epsilon_{i}}
  (\DD_{i} - \AA_{i}  \DD_{i}^{-1} \AA_{i})^{-1},
\quad  
\text{which by Fact \ref{fact}e implies that}
\end{align*}
\begin{align*}
 (\II + \DD_{i}^{-1} \AA_{i})
 \MM_{i+1}^{-1}
 (\II + \AA_{i} \DD_{i}^{-1})
& \Approx{\epsilon_{i}}
 (\II + \DD_{i}^{-1} \AA_{i})
  (\DD_{i} - \AA_{i}  \DD_{i}^{-1} \AA_{i})^{-1}
 (\II + \AA_{i} \DD_{i}^{-1}),
\end{align*}
which by Fact \ref{fact}a implies that
\begin{multline*}
\frac{1}{2}
\left(
\DD_{i}^{-1} +  
 (\II + \DD_{i}^{-1} \AA_{i})
 \MM_{i+1}^{-1}
 (\II + \AA_{i} \DD_{i}^{-1})
 \right)
\\
 \Approx{\epsilon_{i}}
\frac{1}{2}
\left(
\DD_{i}^{-1} +  
 (\II + \DD_{i}^{-1} \AA_{i})
  (\DD_{i} - \AA_{i}  \DD_{i}^{-1} \AA_{i})^{-1}
  (\II + \AA_{i} \DD_{i}^{-1})
 \right)
= \MM_{i}^{-1}.
\qedhere
\end{multline*}
\end{proof}

We now use Lemma~\ref{lem:levelPrecon}, 
  to prove that \textsc{Solve} approximates
  the inverse of $\MM_{0}$.

\begin{lemma}
\label{lem:goodOperator}
Let $\ZZ_{0}$ be the operator defined by \textsc{Solve}.
That is, $\xx_{0} = \ZZ_{0} \bb_{0}$.
Then, 
\[
\ZZ_{0} \Approx{\sum_{i=0}^{d} \epsilon_{i}} \MM_{0}^{-1}.
\]
\end{lemma}

\begin{proof}
Let $\ZZ_{i}$ be the operator such that $\xx_{i} = \ZZ_{i} \bb_{i}$.
That is,
\begin{enumerate}
\item[1.] $  \ZZ_d = \DD_d^{-1}$, and
\item[2.] for  $0 \leq i \leq d - 1$,
$  \ZZ_{i} = \frac{1}{2} \left( \DD_i^{-1} + \left( \II + \DD_i^{-1}\AA_i \right)
\ZZ_{i + 1} \left( \II + \AA_i \DD_i^{-1} \right) \right) $.
\end{enumerate}
We will prove by reverse induction on $i$ that
\[
  \ZZ_i \Approx{\sum_{j = i}^{d} \epsilon_j} \MM_i^{-1}.
\]
The base case of $i = d$ follows from applying Fact~\ref{fact}d to
  $\DD_{d} \Approx{\epsilon_{d}} \MM_{d}$.
For the induction, suppose the result is true for $i + 1$.
Then the induction hypothesis gives
\[
  \ZZ_{i + 1} \Approx{\sum_{j = i + 1}^{d} \epsilon_{j}} \MM_{i + 1}^{-1},
\]
which by Fact~\ref{fact}e gives
\[
  \left( \II + \DD_i^{-1}\AA_i \right) 
  \ZZ_{i + 1} 
  \left( \II + \AA_i \DD_i^{-1} \right) 
\Approx{\sum_{j = i + 1}^{d} \epsilon_j} 
  \left( \II + \DD_i^{-1}\AA_i \right)
  \MM_{i + 1}^{-1}
  \left( \II + \AA_i \DD_i^{-1} \right),
\]
and, by Fact~\ref{fact}a, 
\[
  \ZZ_{i}
\Approx{\sum_{j = i + 1}^{d} \epsilon_j} 
 \frac{1}{2}
 \left(\DD_{i}^{-1} + 
  \left( \II + \DD_i^{-1}\AA_i \right)
  \MM_{i + 1}^{-1}
  \left( \II + \AA_i \DD_i^{-1} \right)
 \right)
\Approx{\epsilon_{i}}
\MM_{i}^{-1},
\]
by Lemma~\ref{lem:levelPrecon}.
An application of Fact~\ref{fact}c completes the induction.
\end{proof}

This leads to a constant factor approximation of $\MM_{0}$.
In order to turn this into a high quality approximation,
  we can use preconditioned Richardson iteration.

\begin{lemma}[Preconditioned Richardson Iteration]
\label{lem:preconRichardson}
 \cite{Saad03:book,Axelsson94:book}
There exists an algorithm $\textsc{PreconRichardson}$ such that
for any symmetric positive semi-definite matrices $\AA$ and $\BB$
such that $\BB \Approx{O(1)} \AA^{-1}$,
and  any error tolerance $0 < \epsilon \leq 1/2$,
\begin{enumerate}
\item Under exact arithmetic,
$\textsc{PreconRichardson}(\AA, \BB, \bb, \epsilon)$ is a linear operator on $\bb$
and if $\ZZ$ is the matrix such that $\ZZ \bb = \textsc{PreconRichardson}(\AA, \BB, \bb, \epsilon)$, then
\[
\ZZ \Approx{\epsilon} \AA^{-1};
\] 
\item $\textsc{PreconRichardson}(\AA, \BB, \bb, \epsilon)$ takes
$O(\log(1/\epsilon))$ iterations, each consisting of one 
  multiplication of a vector by $\AA$ and one by $\BB$.
\end{enumerate}
\end{lemma}

This allows us to solve linear equations in $\MM_0$ to arbitrary precision,
and the overall performance of the solver 
  can be summarized as follows.

\begin{theorem}
\label{thm:solve}
Given an approximate inverse chain for $\MM_0$ where $\DD_i$
and $\AA_i$ have total size $m_i$, 
there is an algorithm $\textsc{Solve}(\MM, \bb, \epsilon)$ that
\begin{enumerate}
\item runs in $O(d \log{n} \log(1/\epsilon))$ depth and
$O(\sum_{i = 0}^{d} m_i\log(1/\epsilon))$ work;
\item is a symmetric linear operator on $\bb$; and
\item if $\ZZ$ is the matrix such that $\textsc{Solve}(\MM, \bb, \epsilon) = \bb$, then $\ZZ \Approx{\epsilon} \MM^{-1}$.
\end{enumerate}
\end{theorem}

\begin{proof}
The algorithm calls the procedure \textsc{Solve} 
inside preconditioned Richardson iteration.
The guarantees of the operator follows from Lemma~\ref{lem:goodOperator},
therefore the guarantees of the overall algorithm is given by Lemma~\ref{lem:preconRichardson}.

To analyze the running time, observe that
  each of the $O(\log(1/\epsilon))$ iterations performs
  one matrix-vector multiplication involving $\MM_0$,
  and then invokes the lower accuracy solver.
This solver in turn 
  performs two matrix-vector products for each matrix $\MM_{i}$,
  along with a constant number of vector additions.
The bounds on the depth and work follow from the fact that the cost
  of each of such a matrix-vector multiplication requires depth $O(\log{n})$ and
  work $O(m_i)$.
\end{proof}

\section{Existence of Sparse Approximate Inverse Chains}\label{sec:exist}
We now show that short sparse approximate inverse chains exist.
Specifically, we show that as long as each $\MM_{i + 1}$ is a good
approximation of $\DD_i - \AA_i \DD_i^{-1} \AA_i$,
$\DD_{i + 1}$ quickly becomes a good approximation to $\MM_{i + 1}$.
The following proposition tells us that the approximation factor
improves by a constant factor if we do not use approximations.

\begin{proposition}\label{pro:ADinvA}
Let $\DD$ and $\AA$ be matrices such that
  $\DD^{-1} \AA$ is diagonalizable and has real eigenvalues.
If all of the eigenvalues of $\DD^{-1} \AA$ are at most
  $1-\lambda$,
Then all of the eigenvalues of
  $\DD^{-1} \AA \DD^{-1} \AA$
  are between $0$ and $(1-\lambda)^{2}$.
\end{proposition}
\begin{proof}
This follows from the fact that $  \DD^{-1} \AA \DD^{-1} \AA = (\DD^{-1} \AA)^{2}$.
\end{proof}

In Appendix~\ref{sec:DhatAhat} we prove the following lemma which
  says that this remains approximately true if we substitute a good approximation
  of $\DD - \AA \DD^{-1} \AA$.
Note that it agrees with Proposition~\ref{pro:ADinvA}
  when $\epsilon  = 0$.

\begin{lemma}\label{lem:DhatAhat}
Let $\DD$ and $\DDhat$ be positive diagonal matrices and
  let $\AA$ and $\AAhat$ be nonnegative symmetric matrices
  such that
 $\DD \pge \AA$, $\DDhat \pge \AAhat$, $\DD \Approx{\epsilon} \DDhat$
  and 
\[
  \DDhat - \AAhat
\Approx{\epsilon}
\DD - \AA \DD^{-1} \AA .
\]
Let the largest eigenvalue of $\DD^{-1} \AA$ be
  $1 - \lambda$.
Then, the eigenvalues of $\DDhat^{-1} \AAhat$
  lie between $1 - \exp (2 \epsilon )$ and
$1 - (1 - (1 - \lambda)^{2}) \exp (-2 \epsilon)$.
\end{lemma}

The following proposition allows us to bound the eigenvalues of
  $\DD_{0}^{-1} \AA_{0}$ in terms of the condition number of $\DD_{0} - \AA_{0}$.

\begin{proposition}\label{pro:ADinv}
Let $\MM = \DD - \AA$ be a positive definite matrix with condition number $\kappa$,
  where $\DD$ is a positive diagonal matrix and $\AA$ is nonnegative.
Then, the eigenvalues of $\DD^{-1} \AA$ are between
  $-1 + 1/\kappa$ and $1 - 1/\kappa$.
\end{proposition}
\begin{proof}
Let $\lambda_{max}$ and $\lambda_{min}$ be the largest and smallest eigenvalues
  of $\MM$.
As the largest eigenvalue of a positive diagonal matrix is its largest entry, and
  the largest eigenvalue of a symmetric matrix is at least its largest entry,
  the largest eigenvalue of $\DD$ is at most $\lambda_{max}$.
We will prove that the smallest eigenvalue of $\II - \DD^{-1} \AA$ is at least
  $\lambda_{min} / \lambda_{max} = 1/\kappa $.
In particular, it is equal to
\[
  \min_{\xx} \frac{\xx^{T} (\II - \DD^{-1/2} \AA \DD^{-1/2}) \xx }{\xx^{T} \xx }
=
  \min_{\xx} \frac{\xx^{T} (\DD -  \AA ) \xx  }{\xx^{T} \DD \xx }
\geq 
  \frac{\lambda_{min} (\DD - \AA)}{\lambda_{max} (\DD)}
\geq 
  \frac{\lambda_{min}}{\lambda_{max}}.
\]
This implies that the largest eigenvalue of $\DD^{-1} \AA$ is at most $1-1/\kappa$.
The Perron-Frobenius theorem tells us that the largest absolute value of an eigenvalue of a nonnegative
  matrix is the largest eigenvalue.
As $\DD^{-1} \AA$ is nonnegative, the bound on its smallest eigenvalue follows.
\end{proof}

Conversely, the following proposition allows us to show that $\DD$ is a good
  approximation of $\DD - \AA$
  if $\DD^{-1} \AA$ is small.

\begin{proposition}\label{pro:DapproxM}
If the eigenvalues of $\DD^{-1} \AA$ lie between $-\alpha$ and $\beta$, then
\[
(1 + \alpha) \DD \pgeq \DD - \AA \pgeq  (1 - \beta ) \DD  .
\]
\end{proposition}
\begin{proof}
Applying parts $a$ and $e$ of Fact~\ref{fact}, we derive
\[
 \beta \II  \pgeq   \DD^{-1/2} \AA \DD^{-1/2} 
 \implies 
  \II - \DD^{-1/2} \AA \DD^{-1/2} \pgeq (1-\beta) \II 
 \implies 
  \DD - \AA  \pgeq (1-\beta) \DD.
\]
The other inequality is similar.
\end{proof}

One obstacle to finding an approximate inverse chain is that the last
  matrix must be approximated by its diagonal.
We use the preceding lemma and propositions to show that
  we can achieve this with a chain whose depth is logarithmic in the
  condition number of $\MM_{0}$.

\begin{corollary}\label{cor:chainShort}
If $\MM_0 = \DD_0 - \AA_0$ is a SDDM matrix with
  condition number $\kappa$,
  $d = \ceil{\log_{4/3} \kappa}$,
  and $\DD_{1}, \dots , \DD_{d}$
  and $\AA_{1}, \dots , \AA_{d}$ satisfy
  conditions $a$ and $b$ of Definition~\ref{def:saic} 
 with
  $\epsilon_{0}, \dots , \epsilon_{d-1} \leq 1/9$,
  then
  $\DD_{d} \Approx{\epsilon_{d}} \MM_{d}$ for
  $\epsilon_{d} = \ln 3$.
\end{corollary}

\begin{proof}
Proposition~\ref{pro:ADinv}  tells us that the eigenvalues of
  $\DD_{0}^{-1} \AA_{0}$ are at most $1 - 1/\kappa$ in absolute value.
For $\epsilon_{i} \leq 1/9$ and $\lambda \leq 1/3$,
\[
  1 - (1 - (1 - \lambda)^{2}) \exp (-2 \epsilon)
\leq 
  1 - (4/3) \lambda .
\]
So, Lemma~\ref{lem:DhatAhat} implies that
  the eigenvalues of $\DD_{d}^{-1} \AA_{d}$
  lie between $1 - \exp (2/9) \geq -1/4$
  and $2/3$.
Proposition \ref{pro:DapproxM} then tells us that
  $\DD \Approx{\gamma} \DD - \AA$,
  where
\[
 \gamma = \max (\ln 3, \ln 5/4) = \ln 3. \qedhere
\]
\end{proof}

It remains to show that we can find sequences of matrices $\DD_{i}$
  and $\AA_{i}$ that satisfy conditions $a$ and $b$ of Definition~\ref{def:saic}.
We begin by proving that 
 as long as $\DD - \AA$ is a SDDM matrix,
  so is $\DD - \AA \DD^{-1} \AA$.

\begin{proposition}
\label{prop:exactGood}
If $\MM = \DD - \AA$ is a SDDM matrix
  with $\DD$ nonnegative diagonal and $\AA$ nonnegative,
  then $\MMhat \defeq  \DD - \AA \DD^{-1} \AA$ is also a SDDM matrix,
  and $\AA \DD^{-1} \AA$ is also nonnegative.
\end{proposition}

\begin{proof}
It is clear that $\MMhat$ is symmetric, and all entries in
$\AA \DD^{-1} \AA$ are nonnegative.
To check that $\MMhat$ is
  diagonally dominant, we compute the sum of the  off-diagonal entries in row $i$:
\[
\sum_{j \neq i} \MMhat_{ij}
= \sum_{j \neq i} \left(\AA \DD^{-1} \AA \right)_{ij} 
= \sum_{j \neq i} \sum_{k} \AA_{ik} \DD_{kk}^{-1} \AA_{kj},
\]
Reordering the two summations and collecting terms gives
\[
= \sum_{k} \AA_{ik} \DD_{kk}^{-1} \left( \sum_{j \neq i} \AA_{kj} \right)
\leq \sum_{k} \AA_{ik}
\leq \DD_{ii}.
\]
As $\MMhat$ is SDD, it is positive semidefinite.
To see that it is positive definite, one need merely observe that it is nonsingular.
This follows from the nonsingularity of $\MM$ and identity \eqref{eqn:identity}.
\end{proof}

This allows us to use the spectral sparsifiers
 of Batson, Spielman, and Srivastava~\cite{BSS} to sparsify
  $\DD - \AA \DD^{-1} \AA$.

\begin{theorem}[Theorem 1.1 from~\cite{BSS}, paraphrased]
\label{thm:bss}
For every $n$-dimensional Laplacian matrix $\LL$
  and every $0 < \epsilon < 1/2$,
  there exists a Laplacian matrix $\LLapprox$
  with $O (n / \epsilon^{2})$ nonzero entries so that
$\LL \Approx{\epsilon} \LLapprox$.
\end{theorem}

\begin{corollary}\label{cor:bss}
For every $n$-dimensional SDDM matrix $\MM = \DD - \AA$ 
  with $\DD$ nonnegative diagonal and $\AA$ nonnegative,
  and every $0 < \epsilon < 1/2$,
 there exist a nonnegative diagonal $\DDhat$ and 
  a nonnegative $\AAhat$ with at most $O (n / \epsilon^{2})$
  nonzero entries so that
\[
\MM \Approx{\epsilon} \DDhat  - \AAhat 
\quad \text{and} \quad 
\DD \Approx{\epsilon} \DDhat .
\]
\end{corollary}
\begin{proof}
Let $\YY $ be the diagonal matrix containing the diagonal entries of $\AA $
  and let $\XX $ be the diagonal matrix so that $\DD - \XX - \AA$ has
  zero row-sums.
Then, $\DD - \XX - \AA $ is a Laplacian matrix, so by
  Theorem~\ref{thm:bss} there exists a Laplacian matrix
  $\DDapprox - \AAapprox$ where $\DDapprox$ is nonnegative diagonal,
  $\AAapprox$ has at most $O (n / \epsilon^{2})$ 
  nonzero entries and is nonnegative with zero diagonal, and
\[
  \DDapprox - \AAapprox 
 \Approx{\epsilon} \DD - \XX - \AA
 = (\DD - \XX - \YY) - (\AA - \YY).
\]  
As neither $\AAapprox$ nor $\AA  - \YY$ have diagonal entries,
  it follows that
\[
  \DDapprox \Approx{\epsilon} \DD - \XX - \YY .
\]
We now set $\DDhat = \DDapprox + \XX + \YY$
  and $\AAhat = \AAapprox + \YY$.
The desired properties of $\DDhat$ and $\AAhat$
  then follow from Fact~\ref{fact}a.
\end{proof}

\begin{theorem}
\label{thm:chainExist}
If $\MM_0 = \DD_0 - \AA_0$ is an $n$-dimensional SDDM matrix with
  condition number at most $\kappa$,
 $\DD_{0}$ is diagonal and $\AA_{0}$ is nonnegative,
 then $\MM_{0}$ has an approximate inverse chain 
  $\DD_{1}, \dots , \DD_{d}$ and $\AA_{1}, \dots , \AA_{d}$
  such that $d = O (\log \kappa)$
  and the total number of nonzero entries in the matrices in the chain
  is  $O(n \log^3{\kappa})$.
\end{theorem}

\begin{proof}
Set $d = \ceil{\log_{4/3} \kappa}$ and $\epsilon_{0}, \dots , \epsilon_{d-1}$
  to the minimum of $1/9$ and $1/2d $.
By Corollary~\ref{cor:bss} there exists a sequence of matrices
  $\MM_{i} = \DD_{i} - \AA_{i}$
  that satisfy conditions $a$ and $b$ of Definition~\ref{def:saic} 
  and that each have $O (n / \log^{2} \kappa )$ nonzero entries.
By Corollary~\ref{cor:chainShort}, we then know that
  $\DD_{d} \Approx{\epsilon_{d}} \MM_{d}$ for $\epsilon_{d} = \ln 3$.
As the sum of the $\epsilon_{i}$ is at most $1/2 + \ln 3 < 2$,
  this sequence of matrices is an approximate inverse chain.
\end{proof}

Theorem~\ref{thm:mainExistence} now follows from
  Theorem~\ref{thm:solve} and Theorem~\ref{thm:chainExist}.

\section{Efficient Parallel Construction}
\label{sec:construct}

We now show how to construct sparse approximate inverse chains efficiently in parallel.
The only obstacle is that we must employ an efficient sparsification routine
  instead of Theorem~\ref{thm:bss}.
We do this through a two-step process.
If $\AA_{i}$ is an $n$-dimensional matrix with $m$ nonzero entries,
  we first compute an approximation of
  $\DD_{i} - \AA_{i} \DD_{i}^{-1} \AA_{i}$
  that has $O (n + m \log n / \epsilon^{2})$
  nonzero entries.
We do this by observing that the off-diagonal entries of this matrix
  come from a sum of weighted cliques, and that we can sparsify
  those cliques individually.
This is made easy by the existence of a closed form for the effective resistance
  between vertices of the cliques.
We then employ a general-purpose sparsification routine
  to further reduce the number of nonzero entries to
  $O (n \log^{c} n / \epsilon^{2})$, for some constant $c$.

To find an algorithm that produces these sparsifiers in nearly-linear work
  and polylogarithmic depth, we look to the original spectral sparsification
  algorithm of Spielman and Teng~\cite{SpielmanTengSparsifier}.
It uses a polylogarithmic number of calls to a graph partitioning routine to
  divide a graph into a small set of edges plus a number of subgraphs of
  high conductance.
It then samples edges from these subgraphs at random.
This algorithm requires nearly-linear work and polylogarithmic depth
  if the graph partitioning algorithm does as well.

The graph partitioning algorithm used in \cite{SpielmanTengSparsifier}
  comes from \cite{SpielmanTengCuts}.
This algorithm, which is based on local graph clustering, could probably be
  implemented efficiently in parallel.
However, it is not stated in a parallel form in that paper.
Instead, we rely on an improvement of this graph partitioning algorithm
  by Orecchia and Vishnoi \cite{OrecchiaVishnoi}.
By using their \textsc{BalCut} algorithm, we obtain a sparsifier with fewer
  edges while using less work and in polylogarithmic depth.
To see that the \textsc{BalCut} algorithm can be implemented in nearly-linear time
  and polylogarithmic depth, we observe that all of its operations are either
  multiplication of vectors by matrices, elementary vector operations, or sorting
  and computing sparsest cuts by sweeps along vectors.
All of these components parallelize efficiently, and  Orecchia and Vishnoi
 showed  that their algorithm only requires $O (m \log^{c} m)$
  work, for some constant $c$, when asked to produce cuts of polylogarithmic conductance.
This is the case in the calls to $\textsc{BalCut}$ made by
Spielman and Teng's \cite{SpielmanTengSparsifier} algorithm,
specifically by the routine \textsc{Partition2} inside \textsc{Sparsify}.

We summarize this discussion with the following theorem.
\begin{theorem}\label{thm:efficientSparsify}
There exists an algorithm that on input an $n$-dimensional Laplacian matrix $\LL$
  with $m$ nonzero entries, an $\epsilon \in [0, 1/2]$,
  produces with probability at least $1-1/n^2$ a Laplacian matrix $\LLhat $ such that
  $\LL \Approx{\epsilon} \LLhat$ and
  $\LLhat$ has $O (n \log^{c} n / \epsilon^{2})$ entries for some constant $c$.
Moreover, this algorithm requires $O (m \log^{c_{1}} n)$ work and $O (\log^{c_{2}} n)$
  depth, for some other constants $c_{1}$ and $c_{2}$.
\end{theorem}

We cannot directly apply the above-described algorithm to sparsify
  $\DD - \AA \DD^{-1} \AA$
  because that matrix could be dense.
So, we must avoid actually constructing this matrix,
and construct  a sparse approximation of it first instead.
Let $m$ be the number of nonzero entries in $\AA$.
We will show how to construct a sparse approximation of 
  $\DD - \AA \DD^{-1} \AA$
  with $O (m \log m / \epsilon^{2})$ nonzero entries.
We begin by writing the second matrix as a sum of outer products:
\[
  \sum_{u} \AA_{:,u} \DD_{u,u}^{-1} \AA_{u,:},
\]
where $\AA_{:,u}$ denotes the $u$th column of $\AA$ and $\AA_{u,:}$ denotes
  the $u$th row.
We can see that the diagonal entries of this sum come from two types
  of products:
  $\AA_{u,u} \DD_{u,u}^{-1} \AA_{u,u}$ and
  $\AA_{u,v} \DD_{v,v}^{-1} \AA_{v,u}$.
There are only $n$ terms of the first type and $m$ terms of the second.
So, all of these can be computed in time $O (n+m)$.
We are more concerned with the off-diagonal entries.
These can come from three types of products:
  $\AA_{u,u} \DD_{u,u}^{-1} \AA_{u,v}$,
  $\AA_{v,u} \DD_{u,u}^{-1} \AA_{u,u}$,
  and $\AA_{v,u} \DD_{u,u}^{-1} \AA_{u,w}$.
There are only $O (m)$ terms of the first two types,
  so we can just compute and store all of them.
For a fixed $u$, the terms of the last type correspond
  to a weighted complete graph on the neighbors of vertex $u$
  in which the edge between $v$ and $w$ has weight
  $\AA_{v,u} \DD_{u,u}^{-1} \AA_{u,w}$.
Denote this graph by $G_{u}$
  and its  Laplacian matrix by $\LL_{u}$.
We now show how to sparsify such a weighted complete graph.

We will do this through the approach of sampling edges by their effective
  resistance introduced by Spielman and Srivastava~\cite{SpielmanSrivastava}.
The effective resistance of an edge $(v,w)$ in a weighted graph $G$ with Laplacian
  matrix $\LL_{u}$ is given by
\[
   (\ee_{v} - \ee_{w}) \LL^{\dag}_{u} (\ee_{v} - \ee_{w}),
\]
where $\ee_{v}$ is the elementary unit vector in direction $v$ and $\LL_{u}^{\dag}$ is the pseudo-inverse of
  $\LL_{u}$.

It has been observed~\cite{KMP1,KelnerL13} that 
  one can strengthen the result of \cite{SpielmanSrivastava} by replacing
  Rudelson's concentration theorem \cite{Rudelson} with that of Rudelson and Vershynin
  \cite{RudelsonVershynin}
  to obtain the following result.

\begin{theorem}\label{thm:sampling}
Let $G = (V,E,\ww)$ be a weighted graph with $n$ vertices
  and let $\LL$ be its Laplacian matrix.
Consider the distribution on edges obtained by choosing an edge from
  $E$ with probability proportional to its weight times its effective resistance,
  and then dividing its weight by this probability.
For every $0 < \epsilon < 1/2$,
  if one forms the Laplacian of 
  $O (n \log n / \epsilon^{2})$ edges chosen from this distribution,
  with replacement, then the resulting Laplacian matrix, $\LLhat$,
  satisfies $\LL \Approx{\epsilon} \LLhat $
  with probability at least $1 - 1/n^{2}$.
\end{theorem}

To apply this theorem, we need to know the effective resistances between pairs
  of vertices in $G_{u}$.
\begin{claim}\label{clm:reff}
The effective resistance between $v$ and $w$ in $G_{u}$
  is 
\[
\frac{\DD_{u,u}}{d_{u}}
\left(\frac{1}{A_{u,v}} + \frac{1}{A_{u,w}} \right),
\]
where $d_{u} = \sum_{v \not = u} A_{u,v}$.
\end{claim}
\begin{proof}
One can check that 
\[
\LL_{u} (\ee_{v}/A_{u,v} - \ee_{w}/A_{u,w}) = 
  (\ee_{v} - \ee_{u}) d_{u} / \DD_{u,u}.
\]
\end{proof}

Note that $\DD_{u,u}$ can be larger than $d_u$.

\begin{corollary}\label{cor:sample}
There exists an algorithm that when given as input
  an $\epsilon \in [0,1/2]$ and 
  an $n$-dimensional SDDM matrix $\MM = \DD - \AA$,
  with $\DD$ nonnegative diagonal and $\AA$ nonnegative
  with $m$ nonzero entries,
  produces with probability at least $1 - 1/n$
  a nonnegative diagonal matrix $\DDhat$ and 
  a nonnegative  $\AAhat$ with at most $O (m \log n / \epsilon^{2})$
  nonzero entries so that
\[
 \DDhat  - \AAhat \Approx{\epsilon}  \DD - \AA \DD^{-1} \AA 
\quad \text{and} \quad 
 \DDhat \Approx{\epsilon} \DD.
\]
Moreover, this algorithm runs in time $O (m \log^{2} n / \epsilon^{2})$.
\end{corollary}
\begin{proof}
We handle the diagonal entries as in Corollary~\ref{cor:bss}.
We also keep all of the off-diagonal entries of the first two types
  described above.
For each vertex $u$, let $\delta_{u}$ be the number of neighbors of vertex $u$.
We will approximate $\LL_{u}$ by a Laplacian $\LLhat_{u}$ with
  $O (\delta_{u} \log n / \epsilon^{2})$ non-zero entries.
By Fact~\ref{fact}b, the sum of the $\LLhat_{u}$
  approximates the sum of the $\LL_{u}$.
By Theorem~\ref{thm:sampling}, we can construct $\LLhat_{u}$
  by sampling the appropriate number of edges of $G_{u}$, and
  then re-scaling them.
It remains to check that we can sample each of these edges 
  in time $O (\log n)$.
To see this, observe that we need to sample edge $(v,w)$ with probability
  proportional to
\[
\AA _{v,u} \DD _{u,u}^{-1} \AA _{u,w} 
( \DD _{u,u} / d_{u} )
\left(1/\AA _{u,v} + 1/A\AA _{u,w} \right)
=
 (\AA _{u,v} + \AA _{u,w}) / d_{u}.
\]
So, in time $O (\delta_{u})$ we can compute for each $v$
  the probability that an edge involving $v$ is sampled.
We can then create a data structure that will allow us to sample
  $v$ with this probability in time $O (\log n)$, and to then follow it by sampling
  $w$ from the distribution induced by fixing the choice of $v$.
\end{proof}

\begin{theorem}
\label{thm:buildChain}
There exists an algorithm that on input 
  an $n$-dimensional diagonal $\DD_{0}$ and  nonnegative $\AA_{0}$
  with $m$ nonzero entries
  such that 
   $\MM_0 = \DD_0 - \AA_0$ 
  has
  condition number at most $\kappa$,
  constructs with probability at least $1/2$ an approximate inverse chain 
  $\DD_{1}, \dots , \DD_{d}$ and $\AA_{1}, \dots , \AA_{d}$
  such that $d = O (\log \kappa)$
  and the total number of nonzero entries in the matrices in the chain
  is  $O(n \log^{c} n\log^3{\kappa})$, for some constant $c$.
Moreover, the algorithm runs in work $O (m \log^{c_{1}} n \log^3{\kappa})$
  and depth $O (\log^{c_{2}} n  \log{\kappa})$, for some other constants $c_{1}$
  and $c_{2}$.
\end{theorem}
\begin{proof}
The proof follows the proof of Theorem~\ref{thm:chainExist}.
However, in place of Corollary~\ref{cor:bss} and Theorem~\ref{thm:bss},
  it uses the algorithms implicit in Corollary~\ref{cor:sample} 
  and Theorem~\ref{thm:efficientSparsify} to construct $\DD_{i}$
  and $\AA_{i}$ from $\DD_{i-1}$ and $\AA_{i-1}$.
\end{proof}

Theorem~\ref{thm:mainEfficient} now follows from Theorem~\ref{thm:solve}
  and Theorem~\ref{thm:buildChain}.

\section*{Acknowledgement}
We thank Gary Miller for very helpful discussions, and Jakub Pachocki
for comments on an earlier draft.

\newcommand{\etalchar}[1]{$^{#1}$}

\begin{appendix}

\section{The Condition Number of a Submatrix}
\label{sec:condition}

To learn more about the reductions between the problems of solving systems of
  linear equations in various forms of SDD matrices, we refer the reader
  to either Section 3.2 of
  \cite{SpielmanTengLinsolve}  or Appendix A of \cite{KOSZ}.
For example, Spielman and Teng \cite{SpielmanTengLinsolve} point out that
  one can solve a system of equations $\LL \xx = \bb$ in a Laplacian matrix by
  first solving the system in which the first entry of $\xx$ is forced to be zero. 
After solving this reduced system, one can orthogonalize the solution with respect
  to the nullspace of $\LL$, the all 1s vector.
The reduced system is a linear system in the submatrix of $\LL$ obtained by removing
  its first row and column.
The resulting matrix is a SDDM matrix.
We now show that its condition number is at most $n$ times the finite condition number of $\LL$.
As the condition number of a submatrix of a positive definite matrix is at most the condition number
  of the original, the same bound holds for every submatrix of $\LL$.

\begin{lemma}\label{lem:cond}
Let $\LL$ be the Laplacian matrix of a connected graph on $n$ vertices and let $\MM$
  be the submatrix of $\LL$ containing all rows and columns of $\LL$ but the first.
Then, $\lambda_{1} (\MM) \geq \lambda_{2} (\LL) / n$.
\end{lemma}
\begin{proof}
For any vector $\vv $ of length $n-1$, let $\uu$ be the vector consisting of a 0 followed by $\vv$.
Then, $\vv^{T} \MM \vv = \uu^{T} \LL \uu$. 
As $\bvec{1}$ spans the nullspace of $\LL$, let $\xx$ be the vector obtained by orthogonalizing
  $\uu$ with respect to $\bvec{1}$:
\[
  \xx = \uu - \mu \bvec{1},
\]
where $\mu$ is the average of the entries of $\uu$.
Then, $\vv^{T} \MM \vv = \xx^{T} \LL \xx$.
We will show that $\norm{\xx} \geq \norm{\vv} / \sqrt{n}$.
The lemma will then follow, as
\[
  \lambda_{1} (\MM) = \min_{\vv} \frac{\vv^{T} \MM \vv}{\vv^{T} \vv}.
\]

We have 
\[
\mu \leq \norm{\uu}_{1}/n = \norm{\vv}_{1} / n \leq (\sqrt{n-1}/n) \norm{\vv}_{2},
\]
so
\[
  \norm{\mu \bvec{1}} \leq \sqrt{(n-1)/n} \norm{\vv}_{2} =  \sqrt{(n-1)/n} \norm{\uu}_{2}.
\]
As $\uu - \mu \bvec{1}$ is orthogonal to $\uu$,
\[
  \norm{\uu - \mu \bvec{1}}^{2} 
 =
  \norm{\uu}^{2} - \norm{\mu \bvec{1}}^{2}
 \geq 
  \norm{\uu}^{2} \left(1 - (n-1)/n \right)
 = 
  \norm{\uu}^{2} /n
 = 
  \norm{\vv}^{2} /n.
\]
\end{proof}

\begin{corollary}\label{cor:cond}
Under the conditions of Lemma~\ref{lem:cond}, the condition number of 
  $\MM$ is at most $n$
  times the finite condition number of $\LL$.
Moreover, the same is true for every principal submatrix of $\LL$.
\end{corollary}
\begin{proof}
The largest and smallest eigenvalues of a principal submatrix of a matrix 
  are between the largest and smallest eigenvalues of that matrix.
As $\MM$ is a sub-matrix of $\LL$, its largest eigenvalue is smaller
  than the largest eigenvalue of $\LL$.
So, the condition number of $\MM$ is at most $n$ times the finite condition number
  of $\LL$.
The condition number of every sub-matrix of $\MM$ can only be smaller.
\end{proof}

\section{Proof of Lemma~\ref{lem:DhatAhat} }
\label{sec:DhatAhat}

The proof is a routine calculation.
We build to it in a few steps.

\begin{proposition}\label{pro:XAX}
Let $\AA$ be a positive definite matrix and let
 $\XX$ be a non-negative diagonal matrix.
Then, 
\[
  \lambda_{max} (\XX \AA \XX) \leq \lambda_{max} (\AA) \lambda_{max} (\XX)^{2},
\]
and
\[
  \lambda_{min} (\XX \AA \XX) \geq \lambda_{min} (\AA) \lambda_{min} (\XX)^{2}.
\]
\end{proposition}
\begin{proof}
We just prove the first inequality.
The second is similar.
We have
\[
 \lambda_{max} (\XX \AA \XX)
=
  \max_{\vv } 
  \frac{\vv^{T} \XX \AA \XX \vv  }{\vv^{T} \vv}
=
  \max_{\vv } 
  \frac{\vv^{T} \AA \vv  }{\vv^{T} \XX^{-2} \vv}
\leq 
  \max_{\vv } 
  \frac{\vv^{T} \AA \vv  \lambda_{max} (\XX)^{2} }{\vv^{T} \vv}
=
  \lambda_{max} (\AA) \lambda_{max} (\XX)^{2}.
\]
\end{proof}

\begin{proposition}\label{pro:DhatAhat}
Let $\DDtil$ and $\DDhat$ be positive diagonal matrices and
  let $\AAtil$ and $\AAhat$ be non-negative symmetric matrices
  such that
 $\DDtil \pge \AAtil$, $\DDhat \pge \AAhat$, $\DDtil \Approx{\epsilon} \DDhat$
  and $\DDtil - \AAtil \Approx{\epsilon} \DDhat - \AAhat$.
Then,
\[
  1 -  \lambda_{max} (\DDhat^{-1} \AAhat )
  \geq 
  ( 1 -   \lambda_{max} (\DDtil^{-1} \AAtil ) ) \exp(-2 \epsilon),
\]
and 
\[
  1 - \lambda_{min} (\DDhat^{-1} \AAhat )
  \leq 
  (1 - \lambda_{min} (\DDtil^{-1} \AAtil ))\exp(2 \epsilon).
\]
\end{proposition}
\begin{proof}
Let $\lambda$ be the largest eigenvalue of $\DDtil^{-1} \AAtil$, and let $\mu = 1- \lambda$.
Then, $\mu $ is the smallest eigenvalue of $\II - \DDtil^{-1/2} \AAtil \DDtil^{-1/2}$.
As 
\[
  \DDhat - \AAhat \pgeq \exp(-\epsilon) (\DDtil - \AAtil),
\]
Fact \ref{fact}e tells us that
\[
  \DDtil^{-1/2} \DDhat \DDtil^{-1/2}
 - 
  \DDtil^{-1/2} \AAhat \DDtil^{-1/2}
\pgeq
  \exp(-\epsilon) (\II - \DDtil^{-1/2}\AAtil \DDtil^{-1/2}).
\]
So, the smallest eigenvalue of
\[
  \DDtil^{-1/2} \DDhat \DDtil^{-1/2}
 - 
  \DDtil^{-1/2} \AAhat \DDtil^{-1/2}
\]
is at least $ \mu \exp(-\epsilon)$.

Let $\XX = \DDtil^{-1/2} \DDhat^{1/2}$.
As $\lambda_{min} (\XX)^{2} \geq \exp(-\epsilon)$,
 Proposition~\ref{pro:XAX} allows us to  conclude that the smallest eigenvalue
  of
\begin{equation}\label{eqn:pro:DhatAhat}
  \II 
 - 
  \DDhat^{-1/2} \AAhat \DDhat^{-1/2}
\end{equation}
is at least
  $\mu \exp(-2\epsilon)$.

We may similarly conclude that the largest eigenvalue
 of \eqref{eqn:pro:DhatAhat} is at most
\[
(1 - \lambda_{min} (\DDtil^{-1} \AAtil )) \exp (2 \epsilon).
\]
\end{proof}

\begin{proof}[Proof of Lemma~\ref{lem:DhatAhat}]
Let $\DDtil = \DD$ and $\AAtil = \AA \DD^{-1} \AA$.
From Proposition~\ref{pro:ADinvA}, we know that the largest eigenvalue
  of $\DDtil^{-1} \AAtil$ is at most $(1 - \lambda)^{2}$, 
  and so 
  $1 - \lambda_{max} (\DDtil^{-1} \AAtil)  \geq   1- (1- \lambda)^{2}$.
Proposition~\ref{pro:DhatAhat} then implies that
\[
  1 - \lambda_{max} (\DDhat^{-1} \AAhat) \geq 
  (1 - (1 - \lambda)^{2})  \exp(-2\epsilon).
\]

As the smallest eigenvalue of $\DDtil^{-1} \AAtil$ is at least zero,
  the largest eigenvalue of $\II - \DDtil^{-1} \AAtil$ is at most $1$.
So, Proposition~\ref{pro:DhatAhat} tells us that the largest eigenvalue
  of $\II - \DDhat^{-1} \AAhat$ is at most $\exp(2\epsilon)$.
This in turn implies that the smallest eigenvalue of
  $\DDhat^{-1} \AAhat$ is at least $1 - \exp(2\epsilon)$.
\end{proof}

\end{appendix}

\end{document}